\documentclass[a4paper,10pt,oneside,final]{article}
\usepackage[italian, english]{babel}
\usepackage[T1]{fontenc}
\usepackage[affil-it]{authblk}
\usepackage[a4paper, left=40mm, right=30mm]{geometry}
\usepackage{graphicx}
\usepackage{mathtools}
\usepackage{afterpage}

\usepackage{setspace}
\usepackage{blindtext}
\usepackage{etex}
\usepackage[utf8x]{inputenc}
\usepackage{amsmath,amssymb,amsfonts,amsthm}

\theoremstyle{plain}
\newtheorem{thm}{Theorem}[section]
\newtheorem{cor}[thm]{Corollary}
\newtheorem{lem}[thm]{Lemma}
\newtheorem{con}[thm]{Conjecture}
\newtheorem{prop}[thm]{Proposition}

\DeclareMathOperator{\rank}{Rank}
\DeclareMathOperator{\sing}{Sing}

\DeclareMathOperator{\cond}{cond}
\DeclareMathOperator{\geom}{geom}
\DeclareMathOperator{\Sym}{Sym}

\DeclareMathOperator{\Aut}{Aut}
\DeclareMathOperator{\PGL}{PGL}

\DeclareMathOperator{\FT}{FT}

\DeclareMathOperator{\Sp}{Sp}
\DeclareMathOperator{\SL}{SL}

\DeclareMathOperator{\Kl}{Kl}
\DeclareMathOperator{\Bi}{Bi}

\theoremstyle{definition}
\newtheorem{defn}{Definition}[section]

\newtheorem{examp}{Example}[section]

\theoremstyle{remark}

\begin{document}

\title{On the size of the maximum of incomplete Kloosterman sums}

\author{Dante Bonolis%
  \thanks{correspondence address: \textit{dante.bonolis@math.ethz.ch}.}}
\date{}
\maketitle
\begin{abstract}
Let $t:\mathbb{F}_{p}\rightarrow\mathbb{C}$ be a complex valued function on $\mathbb{F}_{p}$. A classical problem in analytic number theory is to bound the maximum of the absolute value of the incomplete sum 
\[
M(t):=\max_{0\leq H<p}\Big|\frac{1}{\sqrt{p}}\sum_{0\leq n < H}t(n)\Big|.
\]
In this very general context one of the most important results is the P\'olya-Vinogradov bound 
\[
M(t)\leq \left\|K\right\|_{\infty}\log 3p.
\]
where $K:\mathbb{F}_{p}\rightarrow\mathbb{C}$ is the normalized Fourier transform of $t$. In this paper we provide a lower bound for incomplete Kloosterman sum, namely we prove that for any $\varepsilon >0$ there exists some $a\in\mathbb{F}_{p}^{\times}$ such that
\[
M(e(\tfrac{ax+\overline{x}}{p}))\geq \Big(\frac{1-\varepsilon}{\sqrt{2}\pi}+o(1)\Big)\log\log p.
\]
Moreover we also provide some result on the growth of the moments of $\{M(e(\tfrac{ax+\overline{x}}{p}))\}_{a\in\mathbb{F}_{p}^{\times}}$.
\end{abstract}
\section{Introduction}
Let $t:\mathbb{F}_{p}\rightarrow\mathbb{C}$ be a complex valued function on $\mathbb{F}_{p}$. A classical problem in analytic number theory is to bound the incomplete sums
\[
S(t,H):=\frac{1}{\sqrt{p}}\sum_{0\leq n < H}t(n),
\]
for any $0\leq H< p$. In this very general context one of the most important results is the following:
\begin{thm}[P\'olya-Vinogradov bound, \cite{Po}, \cite{Vi}]
For any $1\leq H<p$ one has 
\[
|S(t,H)|\leq \left\|K\right\|_{\infty}\log 3p,
\]
where $K:\mathbb{F}_{p}\rightarrow\mathbb{C}$ is the normalized Fourier transform of $t$
\[
K(y):=-\frac{1}{\sqrt{p}}\sum_{0\leq x<p}t(x)e\Big(\frac{yx}{p}\Big).
\]
\end{thm}
Notice that if $\left\|K\right\|_{\infty}$ is bounded, then this bound is non-trivial as soon as $H\gg\sqrt{p}\log p$. If one defines
\[
M(t):=\max_{0\leq H<p}|S(t,H)|,
\]
the P\'olya-Vinogradov bound is equivalent to
\[
M(t)\leq \left\|K\right\|_{\infty}\log 3p.
\]
The first question which arises in this setting is the following:  given a function $t:\mathbb{F}_{p}\rightarrow\mathbb{C}$, is the P\'olya-Vinogradov bound sharp for $t$?  And if it is not, what is the best possible bound?\newline

\subsubsection*{Kloosterman sums, Birch Sums and main results}
The aim of this paper is to study the case of the Kloosterman sums and Birch sums. We recall here the definition of these two objects:
\begin{itemize}
\item[$i)$] \textit{Kloosterman sums.} For any $a,b\in\mathbb{F}_{p}^{\times}$ one considers
\[
t:x\mapsto e\Big(\frac{ax+b\overline{x}}{p}\Big)
\]
where $\overline{x}$ denotes the inverse of $x$ modulo $p$. The complete sum over $\mathbb{F}_{p}^{\times}$ of the function above
\[
\text{Kl}(a,b;p):=\frac{1}{\sqrt{p}}\sum_{1\leq x<p}e\Big(\frac{ax+b\overline{x}}{p}\Big)
\]
is called \textit{Kloosterman sum associated to $a,b$}. The Riemann hypothesis over curves for finite fields implies $|\text{Kl}(a,b;p)|\leq 2$ (Weil bound).
\item[$ii)$] \textit{Birch sums.} For any $a,b\in\mathbb{F}_{p}^{\times}$ one considers
\[
t:x\mapsto e\Big(\frac{ax+bx^{3}}{p}\Big).
\]
One defines the \textit{Birch sum associated to $a,b$}
\[
\text{Bi}(a,b;p):=\frac{1}{\sqrt{p}}\sum_{1\leq x<p}e\Big(\frac{ax+bx^{3}}{p}\Big).
\]
Also in this case an application of the Riemann hypothesis over curves for finite field leads to the bound $|\text{Bi}(a,b;p)|\leq 2$ (Weil bound).
\end{itemize} 
It is known that $M(e(\tfrac{ax+\overline{x}}{p}))$ and $M(e(\tfrac{ax+x^{3}}{p}))$ can be arbitrarily large when $a$ varies over $\mathbb{F}_{p}^{\times}$ and $p$ goes to infinity: as a consequence of \cite[Proposition $4.1$]{KS1}, one has that 
\[
\lim_{p\rightarrow\infty}\max_{a} M(e(\tfrac{ax+\overline{x}}{p}))=\lim_{p\rightarrow\infty}\max_{a} M(e(\tfrac{ax+x^{3}}{p}))=\infty.
\]
We will prove the following lower bounds:
\begin{thm}
Let $0<\varepsilon<1$. For all $p$, there exists $S_{p}\subset\mathbb{F}_{p}^{\times}$ such that
\begin{itemize}
\item[$i)$] for any $a\in S_{p}$ one has
\[
M(e(\tfrac{ax+\overline{x}}{p}))\geq \Big(\frac{1-\varepsilon}{\sqrt{2}\pi}+o(1)\Big)\log\log p,
\]
\item[$ii)$] $|S_{p}|\gg_{\varepsilon} p^{1-\frac{\log(4)}{(\log p)^{\varepsilon}}}$.
\end{itemize}
The same is true if one replaces $M(e(\tfrac{ax+\overline{x}}{p}))$ by  $M(e(\tfrac{ax+x^{3}}{p}))$.
\label{thm : kloa}
\end{thm}

Now fix $m\in\mathbb{N}$. For any prime number $p$ coprime with $m$ we denote $H_{m,p}:=\{(x,y)\in\big(\mathbb{F}_{p}^{\times}\big)^{2}:xy=m\}$ and $C_{m,p}:=\{(x,y)\in\big(\mathbb{F}_{p}^{\times}\big)^{2}:x=my^{3}\}$.

\begin{thm}
Let $0<\varepsilon<1$ and fix $m\geq 1$. For all $p$ such that $p\nmid m$, there exists $S_{p}\subset H_{m,p}$ such that
\begin{itemize}
\item[$i)$] for any $(a,b)\in S_{p}$ one has
\[
M(e(\tfrac{ax+b\overline{x}}{p}))\geq \Big(\frac{1-\varepsilon}{\sqrt{2}\pi}+o(1)\Big)\log\log p,
\]
\item[$ii)$] $|S_{p}|\gg_{\varepsilon} p^{1-\frac{\log(4)}{(\log p)^{\varepsilon}}}$.
\end{itemize}
The same is true if one replaces $M(e(\tfrac{ax+b\overline{x}}{p}))$ by  $M(e(\tfrac{ax+bx^{3}}{p}))$ and $H_{m,p}$ by $C_{m,p}$.
\label{thm : klom}
\end{thm}
The proofs of these two Theorems rely on the fact that we can control simultaneously the sign and the size of $\backsim(\log p)^{1-\varepsilon}$ Kloosterman (or Birch) sums. Indeed we will prove
\begin{prop}
For all $p$ there exists $S_{p}\subset\mathbb{F}_{p}^{\times}$ such that for any $a\in S_{p}$
\[
\Kl(an,1;p)\geq\sqrt{2},
\]
for any $1\leq n\leq (\log p)^{1-\varepsilon}$ odd, and
\[
\Kl(an,1;p)\leq-\sqrt{2},
\]
for any $- (\log p)^{1-\varepsilon}\leq n\leq -1$ odd. 
Moreover $|S_{p}|\gg_{\varepsilon}p^{1-\frac{\log(4)}{(\log p)^{\varepsilon}}}$. The same is true if we replace $\Kl$ by $\Bi$.
\label{prop : kl}
\end{prop}
In the second part of the paper, we focus our attention on the growth of the $2k$-th moments of $\{M(e(\tfrac{ax+\overline{x}}{p}))\}_{a\in\mathbb{F}_{p}^{\times}}$ and $\{M(e(\tfrac{ax+x^{3}}{p}))\}_{a\in\mathbb{F}_{p}^{\times}}$ when $p\rightarrow\infty$, getting
\begin{thm}
There exist two absolute positive constants $C>1$ and $c<1$ such that for any fixed $k\geq 1$ and $p\rightarrow\infty$ one has
\[
(c^{2k}+o(1))(\log k)^{2k}\leq\frac{1}{p-1}\sum_{a\in\mathbb{F}_{p}^{\times}}M(e(\tfrac{ax+\overline{x}}{p}))^{2k}\leq ((Ck)^{2k}+o(1))(\log\log p)^{2k},
\]
and for any fixed $m\in\mathbb{Z}\setminus\{0\}$ and $p\rightarrow\infty$
\[
(c^{2k}+o(1))(\log k)^{2k}\leq\frac{1}{p-1}\sum_{(a,b)\in H_{m,p}}M(e(\tfrac{ax+b\overline{x}}{p}))^{2k}\leq ((Ck)^{2k}+o(1))(\log\log p)^{2k}.
\label{thm : klomom}
\]
\label{prop : 2}
\end{thm}

\begin{thm}
There exist two absolute constants $C>1$ and $c<1$ such that for any fixed $k\geq 1$ and $p\rightarrow\infty$ one has
\[
(c^{2k}+o(1))(\log k)^{2k}\leq\frac{1}{p-1}\sum_{a\in\mathbb{F}_{p}^{\times}}M(e(\tfrac{ax+x^{3}}{p}))^{2k}\leq (C^{2k}+o(1))P(k),
\]
and for any fixed $m\in\mathbb{Z}\setminus\{0\}$ and $p\rightarrow\infty$
\[
(c^{2k}+o(1))(\log k)^{2k}\leq\frac{1}{p-1}\sum_{(a,b)\in C_{m,p}}M(e(\tfrac{ax+bx^{3}}{p}))^{2k}\leq (C^{2k}+o(1))P(k),
\]
where $P(k):=\exp(4k\log\log k +k\log\log\log k +o(k))$.
\label{thm : birmom}
\end{thm}
From this we get the following
\begin{cor}
There exist two absolute constants $B,b>0$ such that for $A\rightarrow \infty$ one has
\[
\exp (-\exp(bA))\leq \liminf_{p\rightarrow\infty} \frac{1}{p-1}|\{a\in\mathbb{F}_{p}^{\times}: M((e(\tfrac{ax+x^{3}}{p}))>A\}|\leq \exp \Big(-\exp\Big(BA^{1/2-o(1)}\Big)\Big).
\]
\label{cor : quanti}
\end{cor}

\subsubsection*{Remarks and related works}
\begin{itemize}
\item[$i)$] The upper bound in Theorem $\ref{thm : klomom}$ can be improved conditionally on
\begin{con}[Short sums conjecture for Kloosterman sums]
There exists an $\varepsilon > 0$ such that
\begin{equation}
\Big|\sum_{N\leq x\leq N+H }e\Big(\frac{ax+\overline{x}}{p}\Big)\Big|\ll H^{1-\varepsilon},
\label{eq : cond}
\end{equation}
uniformly for any $1<N<p$, $p^{1/2-\varepsilon/2}<H<p^{1/2+\varepsilon/2}$ and $a\in\mathbb{F}_{p}^{\times}$.
\label{con : short}
\end{con}
Indeed, assuming this conjecture we will prove that
\[
\frac{1}{p-1}\sum_{a\in\mathbb{F}_{p}^{\times}}M(e(\tfrac{ax+\overline{x}}{p}))^{2k}\leq (C^{k}+o(1))P(k).
\]
Notice that Conjecture $\ref{con : short}$ is a (much) weaker form of  Hooley's $R^{*}$-assumption (\cite[page $44$]{Hoo}). In the case of the moments of maximum of incomplete Birch sums we get a better upper bound since the analogue of the $(\ref{eq : cond})$ is known to be true for the function $x\mapsto e\Big(\frac{ax+bx^{3}}{p}\Big)$ (Weyl's inequality).
\item[$ii)$] Notice that for any $(a,b)\in C_{m,p}$ 
\[
\text{Bi}(a,b;p)=\text{Bi}(m,1;p).
\]
Combining Theorem $\ref{thm : klom}$ and Corollary $\ref{cor : quanti}$ one proves that there exist $(a,b),(a',b')\in C_{m,p}$ such that
\[
M(e(\tfrac{ax+bx^{3}}{p}))\gg\log\log p,\quad M(e(\tfrac{a'x+b'x^{3}}{p}))\ll 1
\]
and
\[
\text{Bi}(a,b;p)=\text{Bi}(a',b';p)=\text{Bi}(m,1;p).
\]
Assuming Conjecture $\ref{con : short}$, we can prove the same in the case of Kloosterman sums thanks to the fact that for any $(a,b)\in H_{m,p}$ 
\[
\text{Kl}(a,b;p)=\text{Kl}(m,1;p).
\]
\item[$iii)$] Lamzouri in \cite{Lam} has proved that there exist some (computable) constants $C_{0},C_{1}$ and $\delta$ such that for any $1\ll A\leq \frac{2}{\pi}\log\log p -2\log\log\log p$ one has
\begin{equation}
\frac{1}{p-1}|\{a\in\mathbb{F}_{p}^{\times}: M((e(\tfrac{ax+x^{3}}{p}))>A\}|\geq\text{exp}\Big(-C_{0}\Big(\frac{\pi}{2}A\Big)\Big(1+O\Big(\sqrt{A}e^{-\pi A/4}\Big)\Big)\Big),
\label{eq : lam1}
\end{equation}
He obtains the same result also for incomplete Kloosterman sums. For the family $\{M(e(\tfrac{ax+x^{3}}{p}))\}_{a\in\mathbb{F}_{p}^{\times}}$, he also proved that
\begin{equation}
\frac{1}{p-1}|\{a\in\mathbb{F}_{p}^{\times}: M((e(\tfrac{ax+x^{3}}{p}))>A\}|\leq\text{exp}\Big(-C_{1}\text{exp}\Big(\Big(\frac{\pi}{2}-\delta\Big)A\Big)\Big).
\label{eq : lam2}
\end{equation}
Also in this case the difference between the incomplete Kloosterman sums and incomplete Birch sums depends on the cancellation of the short sums of Kloosterman sums (Conjecture \ref{con : short}). The proof of the lower bound in $(\ref{eq : lam1})$ implies that for at least $p^{1-\frac{1}{\log\log p}}$ elements of $\mathbb{F}_{p}^{\times}$ one has
\begin{equation}
M(\text{Im}(t_{a}))\geq\Big(\frac{2}{\pi}+o(1)\Big)\log\log p,
\label{eq : lam3}
\end{equation}
where $t_{a}=e(\tfrac{ax+x^{3}}{p})$ or $t_{a}=e(\tfrac{ax+\overline{x}}{p})$.
\item[$iv)$] One should compare our result with the case of incomplete character sums. Paley proved that the P\'olya-Vinogradov bound is close to be sharp in this case: indeed in \cite{Pa} is shown that there exist infinitely many primes $p$ such that
\[
M\Big(\Big(\frac{\cdot}{p}\Big)\Big)\gg\log\log p,
\] 
where $\Big(\frac{\cdot}{p}\Big)$ is the Legendre symbol modulo $p$. Similar results were achieved for non-trivial characters of any order by Granville and Soundararajan in \cite{GS}, and by Goldmakher and Lamzouri in \cite{GL1} and \cite{GL2}. On the other hand Montgomery and Vaughan have shown under G.R.H. that
\begin{equation}
M(\chi)\ll\log\log p,
\label{eq : povichar}
\end{equation}
for any $\chi$ (\cite{MV1}), which is the best possible bound up to evaluation of the constant.
\end{itemize}
\textit{Acknowledgment.} I am most thankful to my advisor, Emmanuel Kowalski, for suggesting this problem and for his guidance during these years. I also would like to thank Youness Lamzouri for informing me about his work on sum of incomplete Birch sums.

\subsection{Notation and statement of the main results}
In this section we recall some notion of the formalism of trace functions and state the general version of our main results. For a general introduction on this subject we refer to \cite{FKM5}. Basic statements and references can also be founded in \cite{FKM4}. The main examples of trace functions we should have in mind are
\begin{itemize}
\item[$i)$] For any $f\in\mathbb{F}_{p}[T]$, the function $x\mapsto e(f(x)/p)$: this is the trace function attached to the Artin-Schreier sheaf $\mathcal{L}_{e(f/p)}$.
\item[$ii)$] The Birch sums: $b\mapsto \Bi(a,b;p)$ it can be seen as the trace function attached to the sheaf $\FT (\mathcal{L}_{e((aT^{3})/p)})$
\item[$iii)$] The $n$-th Hyper-Kloosterman sums: the map
\[
x\mapsto\Kl_{n}(x;q):=\frac{(-1)^{n-1}}{q^{(n-1)/2}}\sum_{\substack{y_{1},...,y_{n}\in\mathbb{F}_{q}^{\times}\\y_{1}\cdot...\cdot y_{n}=x}}\psi(y_{1}+\cdots +y_{n}).
\]
can be seen as the trace function attached to the \textit{Kloosterman sheaf} $\mathcal{K}\ell_{n}$ (see \cite{Katz8} for the definition of such sheaf and for its basic properties).
\end{itemize}
\begin{defn} 
Let $\mathcal{F}$ be a middle-extension $\ell$-adic sheaf on $\overline{\mathbb{A}}_{\mathbb{F}_{q}}^{1}$. The \textit{conductor of $\mathcal{F}$} is defined as
\[
c(\mathcal{F}):=\rank(\mathcal{F})+|\sing(\mathcal{F})|+\sum_{x}\text{Swan}_{x}(\mathcal{F}).
\]
\end{defn}
\begin{defn}
Let $p,\ell>2$ be a prime numbers with $p\neq\ell$ and let $r\geq 1$ be an integer. A middle-extension $\ell$-adic sheaf, $\mathcal{F}$, is \textit{$r$-bountiful} if
\begin{itemize}
\item[$i)$] $\mathcal{F}$ is pure of weight $0$ and $\rank (\mathcal{F})\geq 2$,
\item[$ii)$] the geometric and arithmetic monodromy groups of $\mathcal{F}$ satisfy $G_{\mathcal{F}}^{\text{arith}}=G_{\mathcal{F}}^{\geom}$ and $G_{\mathcal{F}}^{\geom}$ is either $\Sp_{r}$ or $\SL_{r}$,
\item[$ii)$] the projective automorphism group
\[
\Aut_{0}(\mathcal{F}):=\{\gamma\in\PGL_{2}(\mathbb{F}_{p}):\gamma^{*}\mathcal{F}\cong\mathcal{F}\otimes\mathcal{L}\text{ for some rank $1$ sheaf}\}
\]
of $\mathcal{F}$ is trivial. 
\end{itemize}
\end{defn}
\begin{defn}
Let $p,\ell>2$ be a prime numbers and let $r\geq 1$ be an integer. A $r$-family $(\mathcal{F}_{a})_{a\in\mathbb{F}_{p}^{\times}}$ is $r$-\textit{acceptable} if the following conditions are satisfied:
\begin{itemize}
\item[$i)$] for any $a\in\mathbb{F}_{p}^{\times}$, $\mathcal{F}_{a}$ is an irreducible middle-extension $\ell$-adic Fourier sheaf on $\mathbb{A}_{\mathbb{F}_{p}}^{1}$ pointwise pure of weight $0$. We denote by $t_{a}$ the trace function attached to $\mathcal{F}_{a}$.
\item[$ii)$] The $\ell$-adic Fourier transform $\text{FT}(\mathcal{F}_{1})$ is an $r$-bountiful sheaf,
\item[$iii)$] for all $y\in\mathbb{F}_{p}$, there exists $\tau_{y}\in\text{PGL}_{2}(\mathbb{F}_{p})$, such that $\tau_{i}\neq\tau_{j}$ if $i\neq j$ and
\[
K_{a}(y)=K_{1}(\tau_{y}\cdot a),
\]
for any $a\in\mathbb{F}_{p}^{\times}$, where $K_{a}(\cdot)$ denote the trace functions attached to $\text{FT}(\mathcal{F}_{a})$.
\end{itemize}
\end{defn}
\begin{defn}
A family of $r$-acceptable families $\mathfrak{F}:=((\mathcal{F}_{a,p})_{a\in\mathbb{F}_{p}^{\times}})_{p}$ is $r$-\textit{coherent} if there exists $C\geq1$ such that
\[
c(\mathcal{F}_{a,p})\leq C,
\] 
for any $p$ prime and $a\in\mathbb{F}_{p}^{\times}$. We call the smallest $C$ with this property the \textit{conductor of the family} and we denote it by $C_{\mathfrak{F}}$.
\end{defn}

\begin{defn}
Let $\mathfrak{F}$ be a $r$-coherent family and for any $A>0$ we define
\[
D_{\mathfrak{F}}(A):=\liminf_{p\rightarrow\infty} \frac{1}{p-1}|\{a\in \mathbb{F}_{p}^{\times}: M(t_{a;p})>A\}|.
\]
\end{defn}

\begin{examp} The following families are $2$-coherent:
\begin{itemize}
\item[$i)$] The family of Artin-Schreier sheaves $\Big(\Big(\mathcal{L}_{e(\frac{ax+\overline{x}}{p})}\Big)_{a\in\mathbb{F}_{p}^{\times}}\Big)_{p}$. Indeed for any $a\in\mathbb{F}^{\times}$, $\mathcal{L}_{e(\frac{ax+\overline{x}}{p})}$ is a middle-extension $\ell$-adic Fourier sheaf pointwise pure of weight $0$ with $\cond(\mathcal{L}_{e(\frac{ax+\overline{x}}{p})})= 2$. Moreover $\text{FT}(\mathcal{L}_{e(\frac{x+\overline{x}}{p})})=\mathcal{K}\ell_{2}$, the Kloosterman sheaf of rank $2$ which is $2$-bountiful (\cite[Paragraph $3.2$]{FKM3}) and
\[
\text{FT}\Big(e\Big(\frac{ax+\overline{x}}{p}\Big)\Big)(y)=\text{Kl}(a+y,1;p),
\]
so we can take $\tau_{y}:=\Big(\begin{matrix} 1 & y \\ 0 & 1 \end{matrix}\Big)$.
\item[$ii)$] The family of Artin-Schreier sheaves $\Big(\Big(\mathcal{L}_{e(\frac{x+b\overline{x}}{p})}\Big)_{b\in\mathbb{F}_{p}^{\times}}\Big)_{p}$. It is enough to argue as above and to observe that
\[
\text{FT}\Big(e\Big(\frac{x+b\overline{x}}{p}\Big)\Big)(y)=\text{Kl}(by,1;p),
\]
so we can take $\tau_{y}:=\Big(\begin{matrix} y & 0 \\ 0 & 1 \end{matrix}\Big)$.
\item[$iii)$] Fix $m\in\mathbb{Z}$. The family of Artin-Schreier sheaves $\Big(\Big(\mathcal{L}_{e(\frac{ax+m\overline{ax}}{p})}\Big)_{a\in\mathbb{F}_{p}^{\times}}\Big)_{p}$ is $2$-coherent. Also in this case one argues as above and observes that
\[
\text{FT}\Big(e\Big(\frac{ax+m\overline{ax}}{p}\Big)\Big)(y)=\text{Kl}(my+m\overline{a},1;p),
\] 
so we can take $\tau_{y}:=\Big(\begin{matrix} my & m \\ 1 & 0 \end{matrix}\Big)$.
\item[$iv)$] With similar arguments one shows that the families $\Big(\Big(\mathcal{L}_{e(\frac{ax+x^{3}}{p})}\Big)_{a\in\mathbb{F}_{p}^{\times}}\Big)_{p}$, $\Big(\Big(\mathcal{L}_{e(\frac{x+bx^{3}}{p})}\Big)_{b\in\mathbb{F}_{p}^{\times}}\Big)_{p}$ and $\Big(\Big(\mathcal{L}_{e(\frac{ax+m(xa)^{3}}{p})}\Big)_{a\in\mathbb{F}_{p}^{\times}}\Big)_{p}$ are $2$-coherent families. 
\end{itemize}
\end{examp}
Then Theorem $\ref{thm : kloa}$ and $\ref{thm : klom}$ are consequences of the following  
\begin{thm}
Let $0<\varepsilon <1$. Let $\mathfrak{F}=((\mathcal{F}_{a,p})_{a\in\mathbb{F}_{p}^{\times}})_{p}$ be a $2$-coherent. For all $p$ there exists $S_{p}\subset\mathbb{F}_{p}^{\times}$ such that
\begin{itemize}
\item[$i)$] for any $a\in S_{p}$ one has
\[
M(t_{a,p})\geq \Big(\frac{1-\varepsilon}{\sqrt{2}\pi}+o(1)\Big)\log\log p,
\]
\item[$ii)$] $|S_{p}|\gg_{\varepsilon, C_{\mathfrak{F}}} p^{1-\frac{\log(4)}{(\log p)^{\varepsilon}}}$.
\label{thm : tracebound}
\end{itemize}
\end{thm}

Similarly, Theorem $\ref{thm : klomom}$ and $\ref{thm : birmom}$ are consequence of:
\begin{thm}
Let $\mathfrak{F}=((\mathcal{F}_{a,p})_{a\in\mathbb{F}_{p}^{\times}})_{p}$ be a $r$-coherent family. There exist two positive constant $C>1$ and $c<1$ depending only on $c_{\mathfrak{F}}$ such that for any fixed $k\geq 1$ 
\[
(c^{2k}+o(1))(\log k)^{2k}\leq\frac{1}{p-1}\sum_{a\in\mathbb{F}_{p}^{\times}}M(t_{a,p})^{2k}\leq ((Ck)^{2k}+o(1))(\log\log p)^{2k}.
\]
If moreover one has that there exists an $\varepsilon >0$ such that
\begin{equation}
\Big|\sum_{N\leq x\leq N+H }t_{a,p}(x)\Big|\ll_{c_{\mathfrak{F}}} H^{1-\varepsilon}
\label{eq : cond1}
\end{equation}
uniformly for any $1<N<p$, $p^{1/2-\varepsilon/2}<H<p^{1/2+\varepsilon/2}$ and $a\in\mathbb{F}_{p}^{\times}$, then for any fixed $k\geq 1$ one has
\[
(c^{2k}+o(1))(\log k)^{2k}\leq\frac{1}{p-1}\sum_{a\in\mathbb{F}_{p}^{\times}}M(t_{a,p})^{2k}\leq (C^{2k}+o(1))P(k).
\]
\label{thm : tracemom}
\end{thm}
We then get the following
\begin{cor}
Same notation as in Theorem $\ref{thm : tracemom}$. Then:
\begin{itemize}
\item[$i)$] for any $A>0$ one has
\[
D_{\mathfrak{F}}(A)\geq \exp (-\exp(bA)),
\]
where $b>0$ depends only on $c_{\mathfrak{F}}$,
\item[$ii)$] if the condition $(\ref{eq : cond1})$ holds, there exists a $B>0$ depending only on $c_{\mathcal{F}}$ such that for $A\rightarrow\infty$ one has
\[
D_{\mathfrak{F}}(A)\leq \exp \Big(-\exp\Big(BA^{1/2-o(1)}\Big).
\]
\end{itemize}
\label{cor : quan}
\end{cor}

\section{Proof of Theorem \ref{thm : tracebound}}
\subsubsection*{First step: Fourier expansion and F\'ejer Kernel}
The first step for both Theorems is to get a quantitative version of the Fourier expansion for $\frac{1}{\sqrt{p}}\sum_{x\leq \alpha p}t(x)$:
\begin{lem}
Let $t:\mathbb{F}_{p}\rightarrow\mathbb{C}$ be a complex valued function on $\mathbb{F}_{p}$, then for any for any $0<\alpha<1$ we have
\[
\begin{split}
\frac{1}{\sqrt{p}}\sum_{x\leq \alpha p}t(x)=&-\frac{1}{2\pi i}\sum_{1\leq |n|\leq N}\frac{K(n)}{n}(1-e(-\alpha n))+\alpha K(0)+ O\Big(\frac{\left\| t\right\|_{\infty}\sqrt{p}\log p}{N}\Big),
\end{split}
\]
for any $1\leq N\leq p$, where the implied constant is absolute.
\end{lem} 
\begin{proof}
We use the same strategy used in \cite{Po}. Let us introduce the function
\[
\Phi(s)=
\begin{cases}
1			&\text{if $0<s<2\pi\alpha$},\\
\frac{1}{2}	&\text{if $s=0$ or $s=2\pi\alpha$},\\
0			&\text{if $2\pi\alpha<s<2\pi$}.\\
\end{cases}
\] 
Then the Fourier series of $\Phi$ is
\[
\begin{split}
\Phi(s)&=\alpha+\sum_{n>0}\frac{\sin 2\pi\alpha n}{\pi n}\cos(ns)-\frac{\cos 2\pi\alpha n-1}{n\pi}\sin(ns)\\& =\alpha+\frac{1}{\pi}T(s)-\frac{1}{\pi}T(s-2\pi\alpha),
\end{split}
\] 
where
\[
T(x):=\sum_{n>0}\frac{\sin nx}{n}.
\]
Observe that for any $N>1$ one has
\[
T(x)=\sum_{0<n\leq N}\frac{\sin nx}{n}+R_{N}(x).
\]
with $R_{N}(0)=R_{N}(\pi)$, $R_{N}(2\pi-x)=-R_{N}(x)$ and $|R_{N}(x)|=O(1/Nx)$ for any $x\in (0,\pi]$ \cite[eq. $10$]{Po}. Then we have
\[
\begin{split}
\frac{1}{\sqrt{p}}\sum_{x\leq \alpha p}t(x)&=\frac{1}{\sqrt{p}}\sum_{x<p}t(x)\Phi\Big(\frac{2\pi x}{p}\Big)+O(\left\| t\right\|_{\infty}/\sqrt{p})\\&=\frac{1}{\sqrt{p}}\sum_{x<p}t(x)\Big(\alpha +\frac{1}{\pi}T\Big(\frac{2\pi x}{p}\Big)-\frac{1}{\pi}T\Big(\frac{2\pi x}{p}-2\pi\alpha\Big)\Big)+O\Big(\frac{\left\| t\right\|_{\infty}}{\sqrt{p}}\Big)\\&=\frac{1}{\sqrt{p}}\sum_{x<p}t(x)\Bigg(\alpha + \frac{1}{\pi}\sum_{0<n\leq N}\frac{\sin \Big(\frac{2\pi nx}{p}\Big)}{n}+\frac{1}{\pi}R_{N}\Big(\frac{2\pi x}{p}\Big)\\&-\frac{1}{\pi}\sum_{0<n\leq N}\frac{\sin \Big(\frac{2\pi nx}{p}-2\pi\alpha n\Big)}{n}+\frac{1}{\pi}R_{N}\Big(\frac{2\pi x}{p}-2\pi\alpha \Big)\Bigg)+O\Big(\frac{\left\| t\right\|_{\infty}}{\sqrt{p}}\Big)\\&=\frac{1}{\sqrt{p}}\sum_{x<p}t(x)\Big(\frac{1}{\pi}\sum_{0<n\leq N}\frac{\sin \Big(\frac{2\pi nx}{p}\Big)}{n}-\frac{1}{\pi}\sum_{0<n\leq N}\frac{\sin \Big(\frac{2\pi nx}{p}-2\pi\alpha n\Big)}{n}\Big)\\&+\alpha K(0)+O\Big(\frac{\left\| t\right\|_{\infty}\sqrt{p}\log p}{N}\Big).
\end{split}
\]
On the other hand we have
\[
\sin \Big(\frac{2\pi nx}{p}\Big)=\frac{e\Big(\frac{nx}{p}\Big)-e\Big(-\frac{nx}{p}\Big)}{2i}
\]
and 
\[
\sin \Big(\frac{2\pi nx}{p}-2\pi\alpha\Big)=\frac{e\Big(\frac{ nx}{p}-\alpha n\Big)-e\Big(-\Big(\frac{nx}{p}-\alpha n\Big)\Big)}{2i}.
\]
Then one has
\[
\frac{1}{\sqrt{p}}\sum_{x\in\mathbb{F}_{p}}t(x)\Big(\frac{e\Big(\frac{ nx}{p}\Big)-e\Big(-\frac{ nx}{p}\Big)}{2i}\Big)=-\frac{1}{2i}(K(n)-K(-n))
\]
and similarly
\[
\frac{1}{\sqrt{p}}\sum_{x\in\mathbb{F}_{p}}t(x)\Big(\frac{e\Big(\frac{ nx}{p}-\alpha n\Big)-e\Big(-\Big(\frac{nx}{p}-\alpha n\Big)\Big)}{2i}\Big)=-\frac{1}{2i}(e(-\alpha n)K(n)-e(\alpha n)K(-n)).
\]

\end{proof}

Now we use the same strategy of \cite{Pa} introducing the Fej\'er's kernel:
\begin{lem}
For any $t:\mathbb{F}_{p}\rightarrow\mathbb{C}$ one has 
\[
M(t)\geq\max_{\substack{\alpha\in[0,1]\\1\leq N<p}}\Big|\frac{1}{4\pi}\sum_{1\leq |n| < N}\frac{K(n)}{n}(1-e(-\alpha n))\Big| +O(\left\|K\right\|_{\infty})
\]
\label{lem : ker}
\end{lem}
\begin{proof}
The quantitative version of the Fourier transform leads to 
\[
\begin{split}
\frac{1}{\sqrt{p}}\sum_{x\leq \alpha p}t(x)&=-\frac{1}{2\pi i}\sum_{1\leq |n|\leq p}\frac{K(n)}{n}(1-e(-\alpha n))+\alpha K(0)+ O(1)\\&=-\frac{1}{2\pi }\sum_{1\leq |n|< p}\frac{K(n)}{n}(1-e(-\alpha n))+O(\left\|K\right\|_{\infty}),
\end{split}
\]
at this point we extend the outer sum to all values modulo $p$ using the Fej\'er's kernel: for any $1<N<p$ we have
\begin{equation}
\begin{split}
\frac{1}{2\pi i}\sum_{1\leq |n|\leq N}\frac{K(n)}{n}(1-e(-\alpha n))&= \frac{1}{2\pi i}\sum_{1\leq |n|\leq p}\frac{K(n)}{n}(1-e(-\alpha n))\times\\&\times \sum_{1<|a|\leq N}\phi (a)\int_{0}^{1}e((a-n)\vartheta)d\vartheta+O(\left\|K\right\|_{\infty})\\&=\int_{0}^{1}A_{\theta}\Phi_{N}(\vartheta)d\vartheta\\&+O(\left\|K\right\|_{\infty}),
\end{split}
\end{equation}
where
\begin{equation}
\phi (a):=1-\frac{|a|}{N},\qquad\Phi_{N} (\vartheta):=\sum_{|a|\leq N}\phi (a)e(a\vartheta)=\frac{1}{N}\Big(\frac{\sin\frac{N\vartheta}{2}}{\sin\frac{\vartheta}{2}}\Big)^{2},
\end{equation}
is the F\'ejer Kernel, and 
\[
A_{\theta}:=\frac{1}{2\pi i}\sum_{1\leq |n|\leq p}\frac{K(n)}{n}(1-e(-\alpha n))e(-\vartheta n).
\]
On the other hand the triangular inequality leads to
\[
\begin{split}
\max_{\vartheta\in[0,1]}|A_{\theta}|&\leq 2 \max_{\alpha\in[0,1]}\Big|\frac{1}{2\pi i}\sum_{1\leq |n|\leq p}\frac{K(n)}{n}(1-e(-\alpha n))\Big|\\&\leq 2\max_{\alpha\in[0,1]}\Big|\frac{1}{\sqrt{p}}\sum_{x\leq \alpha p}t(x)\Big|+O(\left\|K\right\|_{\infty})\\&\leq 2M(t)+O(\left\|K\right\|_{\infty})
\end{split}
\]
So we obtain the bound
\begin{equation}
\Big|\frac{1}{4\pi i}\sum_{1\leq |n|\leq N}\frac{K(n)}{n}(1-e(-\alpha n))\Big|\leq\Big(M (t)+O(\left\|K\right\|_{\infty})\Big)\cdot\int_{0}^{1}\Phi_{N}(\vartheta)d\vartheta.
\end{equation}
On the other hand using the fact
\[
\int_{0}^{1}\Phi_{N}(\vartheta)d\vartheta=1
\]
we conclude the proof.
\end{proof}

To conclude, it is enough to prove the following
\begin{prop}
Same assumption as in Theorem $\ref{thm : tracebound}$. Let $0<\varepsilon<1$. Then for all $p$ there exists $S_{p}\subset\mathbb{F}_{p}^{\times}$ such that for any $a\in S_{p}$
\[
K_{1,p}(\tau_{n}\cdot a)\geq\sqrt{2},
\]
for any $1\leq n\leq (\log p)^{1-\varepsilon}$ odd, and:
\[
K_{1,p}(\tau_{n}\cdot a)\leq-\sqrt{2},
\]
for any $- (\log p)^{1-\varepsilon}\leq n\leq -1$ odd. 
Moreover $|S_{p}|\gg_{\varepsilon, c_{\mathfrak{F}}}p^{1-\frac{\log(4)}{(\log p)^{\varepsilon}}}$.
\label{lem : kl}
\end{prop}
Assuming this Proposition, which we prove in the next section, let us prove Theorem $\ref{thm : tracebound}$. We have that
\[
\begin{split}
M(t_{a,p})&=  \max_{\alpha\in [0,1]}\Big|\frac{1}{\sqrt{p}}\sum_{x\leq \alpha p}t_{a,p}(x)\Big|\\&\geq\frac{1}{4\pi}\max_{\substack{\alpha\in[0,1],\\1\leq N< p}}\Big|\sum_{1\leq |n|\leq N}\frac{K_{a,p}(n)}{n}(1-e(-\alpha n))\Big|+O(1)\\&=\frac{1}{4\pi}\max_{\substack{\alpha\in[0,1],\\1\leq N< p}}\Big|\sum_{1\leq |n|\leq N}\frac{K_{1,p}(\tau_{n}\cdot a)}{n}(1-e(-\alpha n))\Big|+O(1)\\&\geq \frac{1}{4\pi}\Big|\sum_{1\leq |n|\leq (\log p)^{1-\varepsilon}}\frac{K_{1,p}(\tau_{n}\cdot a)}{n}(1+(-1)^{n+1})\Big|+O(1)\\&\geq \frac{2\sqrt{2}}{4\pi}\sum_{\substack{1\leq |n|\leq (\log p)^{1-\varepsilon}\\n\equiv 1 (2)}}\frac{1}{n}+O(1)\\&\geq\Big(\frac{1-\varepsilon}{\sqrt{2}\pi}+o(1)\Big)\log\log p.
\end{split}
\]
for any $a\in S_{p}$, where in the second step uses the fact that $K_{a,p}(n)=K_{1,p}(\tau_{n}\cdot a)$ (the family is $2$-bountiful).

\subsubsection*{Proof of Lemma $\ref{lem : kl}$ via Chebyshev Polynomials}
From now on $p$ is a fixed prime number. We consider an irreducible $2$-bountiful sheaf $\mathcal{K}$ on $\overline{\mathbb{A}}_{\mathbb{F}_{p}}^{1}$ and we will denote the trace function attached to it by $K(\cdot)$. The $2$-bountiful condition on the sheaf $\mathcal{K}$ implies that for any $a\in\mathbb{F}_{p}$, one has
\[
K(a)=2\cos(\theta (a)).
\]
with $\theta (a)\in [0,\pi]$. We call $\theta (a)$ the angle associated to $K(a)$. We recall that there exist polynomials $U_{n}$ for $n\geq 0$ such that
\[
U_{n}(2\cos\theta)=\frac{\sin ((n+1)\theta)}{\sin\theta},
\]
for all $\theta\in [0,\pi]$. In terms of Representation Theory, these are related to the characters of the symmetric power of the standard representation of $\text{SU}_{2}$. In particular by Peter-Weyl Theorem, these form an orthonormal basis of $L^{2}([0,\pi],\mu_{\text{ST}})$. Note that we can see $U_{n}(K(\cdot))$ as the trace function attached to the sheaf $\Sym^{n}(\mathcal{K})$. Moreover we call \textit{trigonometric polynomial of degree $s\geq 0$} any $Y\in L^{2}([0,\pi],\mu_{\text{ST}})$ written in the form
\[
Y=\sum_{i=0}^{s}y(i)U_{i}.
\]
with $y(s)\neq 0$. Let us start by proving some property of the sheaf $\Sym^{n}(\mathcal{K})$:
\begin{lem}
Let $\mathcal{K}$ as above. For any $n>0$:
\begin{itemize}
\item[$i)$] The geometric monodromy group of $\Sym^{n}(\mathcal{K})$ is given by
\[
G_{\Sym^{n}(\mathcal{K})}^{\geom}\cong
\begin{cases}
\text{SU}_{2}			&\text{if $n$ is odd,}\\
\text{SU}_{2}/\{\pm 1\}	&\text{if $n$ is even.}
\end{cases}
\]
\item[$ii)$] The projective automorphism group
\[
\begin{split}
\Aut_{0}(\Sym^{n}(\mathcal{K})):=\{\gamma\in\PGL_{2}(\mathbb{F}_{p}):&\gamma^{*}\Sym^{n}(\mathcal{K})\cong\Sym^{n}(\mathcal{K})\otimes\mathcal{L}\\& \text{ for some rank $1$ sheaf }\mathcal{L}\}
\end{split}
\]
is trivial.
\item[$iii)$] The conductor of $\Sym^{n}(\mathcal{K})$ is bounded by 
\[
c(\Sym^{n}(\mathcal{K}))\leq n \cdot c (\mathcal{K}).
\]
\end{itemize}
\label{lem : sym}
\end{lem} 
\begin{proof} 
Let us start with part $(i)$: by the definition of the geometric monodromy one has that $G_{\Sym^{n}(\mathcal{K})}^{\geom}=\Sym^{n}(G_{\mathcal{K}}^{\geom})$. Then the result follows because $G_{\mathcal{K}}^{\geom}=\text{SU}_{2}$ by hypothesis. Let us prove now part $(ii)$. Let $\gamma\in\PGL_{2}(\mathbb{F}_{p})$. First observe that
\[
t_{\Sym^{n}(\mathcal{K})}(x)=\frac{\sin((n+1)\theta (x))}{\sin(\theta (x))},\qquad t_{\gamma^{*}\Sym^{n}(\mathcal{K})}(x)=\frac{\sin((n+1)\theta (\gamma\cdot x))}{\sin(\theta (\gamma\cdot x))},
\]
where $t_{\mathcal{K}}(x)=2\cos\theta (x)$. Thanks to the fact that $\mathcal{K}$ is a bountiful sheaf we know that the angles $\{(\theta (x),\theta (\gamma\cdot x)): x\in\mathbb{F}_{p^{r}}\}$ become equidistributed in $([0,\pi]\times [0,\pi],\mu_{\text{ST}}\otimes\mu_{\text{ST}})$ when $r\rightarrow\infty$ (Goursat-Kolchin-Ribet criterion). By contradiction, assume that $\gamma^{*}\Sym^{n}(\mathcal{K})\cong\Sym^{n}(\mathcal{K})\otimes\mathcal{L}$ for some rank $1$ sheaf. We may assume that $\mathcal{L}$ is of weights $0$. Let $U$ be a dense open set where $\gamma^{*}\Sym^{n}(\mathcal{K}),\Sym^{n}(\mathcal{K})$ and $\mathcal{L}$ are lisse. Using the equidistribution  we can find $x\in U$  such that  
\[
t_{\Sym^{n}(\mathcal{K})}(x)<1/4,\qquad t_{\gamma^{*}\Sym^{n}(\mathcal{K})}(x)>3/4.
\]
On the other hand in $U$ one would have
\[
|t_{\Sym^{n}(\mathcal{K})}(x)|=|t_{\gamma^{*}\Sym^{n}(\mathcal{K})}(x)|.
\]
and this is absurd. Part $(iii)$ is just a consequence of Deligne's Equidistribution Theorem (see for example \cite[Paragraph $3.6$]{Katz8}).
\end{proof}
\begin{lem}
Let $(Y_{i})_{i=0}^{n}$ be a family of trigonometric polynomials as above such that for any $i$, $\deg Y_{i}\leq d$, and let $(\tau_{i})_{i=1}^{n}\in\PGL_{2}(\mathbb{F}_{p})$ such that $\tau_{i}\neq\tau_{j}$ if $i\neq j$ then
\begin{equation}
\Big|\sum_{a\in\mathbb{F}_{q}^{\times}}\prod_{i=0}^{n}Y_{i}(\theta(\tau_{i}\cdot a))-p\prod_{i=0}^{n}y_{i}(0)\Big|\leq nC^{n} c(\mathcal{K})^{2}d^{2n+2}y^{n}\sqrt{p},
\end{equation}
where $y=\max_{i,j}|y_{i}(j)|$ and the constant $C$ is absolute.
\label{lem : prod2}
\end{lem}
\begin{proof}
To prove the lemma it is enough to bound
\[
S=\sum_{a\in\mathbb{F}_{q}^{\times}}\prod_{i=0}^{n}U_{n_{i}}(K(\tau_{i}\cdot a)),
\]
when at least one of the $n_{i}\neq 0$. Thanks to Lemma $\ref{lem : sym}$ and \cite[Paragraph $3.1$]{FKM3} we can apply \cite[Theorem $2.7$]{FKM3} getting
\[
\Big|S-p\prod_{i=0}^{n}M_{n_{i}}\Big|\leq L\sqrt{p},
\]
where 
\[
L:=C'n\cdot(\max_{n_{i}}\rank (\Sym^{n_{i}}(\mathcal{K}))^{n}\cdot(\max_{n_{i}} c (\Sym^{n_{i}}(\mathcal{K}))^{2},
\]
with $C'$ absolute constant (see \cite[Ptoposition $4.4$]{Per}), and for any $n_{i}$
\[
M_{n_{i}}:=\text{Mult}(1,\Sym^{n_{i}}\text{Std})=
\begin{cases}
1	&\text{if $n_{i}=0$,}\\
0	&\text{otherwise,}
\end{cases}
\]
and $\text{Std}$ denotes the standard representation on $\text{SU}_{2}$. 
Then if at least one of the $n_{i}\neq 0$ we have
\[
|S|\leq L\sqrt{p}.
\]
The result then follows from the fact that
\[
\rank\Sym^{n_{i}}(\mathcal{K}))=n_{i}+1\leq 2d,\quad c (\Sym^{n_{i}}(\mathcal{K}))\leq n_{i}c(\mathcal{K})\leq d c(\mathcal{K}),
\]
because $n_{i}\leq d$ for all $i$ by assumption.
\end{proof}

\subsubsection*{Proof of Proposition $\ref{lem : kl}$}
We can now prove Proposition $\ref{lem : kl}$. Let $z\in\mathbb{N}$ be an odd positive number and let $\gamma\in\mathbb{N}$, we denote by $\boldsymbol{\theta}:=(\theta_{2j-z})_{j=0}^{z}\in [0,\pi]^{z+1}$ and by $\chi_{\frac{1}{\gamma}}(\cdot)$ (resp. $\chi_{-\frac{1}{\gamma}}(\cdot)$) the characteristic function of $[0,\frac{\pi}{2}-\frac{\pi}{\gamma}]$ (resp.$[\frac{\pi}{2}+\frac{\pi}{\gamma},\pi]$). To prove Proposition $\ref{lem : kl}$ we start approximating the function
\[
\prod_{\substack{i=1\\i\equiv 1 (2)}}^{z}\chi_{\frac{1}{\gamma}}(\theta (\tau_{i}\cdot a))\prod_{\substack{i=1\\i\equiv 1 (2)}}^{z}\chi_{-\frac{1}{\gamma}}( \theta(\tau_{-i}\cdot a))
\]
using Chebyshev polynomials. We use the same method adopted in \cite[Section $3$]{KLSW}: for any $z$, we find an integer $L\equiv -1 \mod 2\gamma$ and two families of trigonometric polynomials $\{\alpha_{i}\}$, and $\{\beta_{i}\}$ such that if we define
\[
A_{L}\Big(\frac{\boldsymbol{\theta}}{\pi} \Big):=\prod_{1\leq|i|\leq z}\alpha_{L,i}\Big(\frac{\theta_{i}}{\pi} \Big)\qquad B_{L}\Big(\frac{\boldsymbol{\theta}}{\pi} \Big)=\sum_{1\leq |i|\leq z}\beta_{L,i}\Big(\frac{\theta_{i}}{\pi} \Big)\prod_{j\neq i}\alpha_{L,j}\Big(\frac{\theta_{i}}{\pi} \Big),
\]
the following inequality holds
\begin{equation}
A_{L}\Big(\frac{\boldsymbol{\theta}}{\pi} \Big)-B_{L}\Big(\frac{\boldsymbol{\theta}}{\pi} \Big)\leq \prod_{\substack{i=1\\i\equiv 1 (2)}}^{z}\chi_{\frac{1}{\gamma}}(\theta_{i})\prod_{\substack{i=1\\i\equiv 1 (2)}}^{z}\chi_{-\frac{1}{\gamma}}(\theta_{-i}),
\end{equation}
for any $\boldsymbol{\theta}\in[0,\pi]^{z+1}$. Moreover we will prove
\begin{lem}
With the notation as above, we have:
\begin{itemize}
\item[$i)$] There exist two constant $L_{0}\geq 1$ and $c>0$ depending only on $\gamma$, such that the contribution $\Delta$ of the constant term in the Chebyshev expansions of $A_{L}\big(\frac{\boldsymbol{\theta}}{\pi} \big)-B_{L}\big(\frac{\boldsymbol{\theta}}{\pi} \big)$ satisfies:
\[
\Delta\geq\frac{1}{2}\Big(\frac{1}{2}-\frac{1}{\gamma}\Big)^{z+1},
\]
if $L$ is the smallest integer such that $L\equiv -1\mod 2\gamma$ satisfying $L\geq\max (cz,L_{0})$.
\item[$ii)$] All coefficients in the Chebyshev expansion of the factors in $A_{L}\big(\frac{\boldsymbol{\theta}}{\pi} \big)$ and the terms in $B_{L}\big(\frac{\boldsymbol{\theta}}{\pi} \big)$ are bounded by $1$.
\item[$iii)$] The degrees, in terms of the Chebyshev expansion, of the factors of $A_{L}\big(\frac{\boldsymbol{\theta}}{\pi} \big)$ and $B_{L}\big(\frac{\boldsymbol{\theta}}{\pi} \big)$ are $\leq 2L$.
\end{itemize}
\label{lem : tec}
\end{lem}
Once we have this Lemma we can easily get Proposition $\ref{lem : kl}$. Fix $\gamma=\frac{1}{4}$ in Lemma $\ref{lem : tec}$ and denote $S_{p}$ the set of $a\in\mathbb{F}_{q}^{\times}$ which satisfy the property in the Proposition $\ref{lem : kl}$. Let $L$ be as in part $(i)$ of Lemma $\ref{lem : tec}$, then we have
\begin{equation}
\begin{split}
|S_{p}|=&\sum_{a\in\mathbb{F}_{p}^{\times}}\prod_{\substack{i=1\\i\equiv 1 (2)}}^{z}\chi_{\frac{1}{4}}(\theta (\tau_{i}\cdot a))\prod_{\substack{i=1\\i\equiv 1 (2)}}^{z}\chi_{-\frac{1}{4}}( \theta (\tau_{-i}\cdot a))\\&\geq\sum_{a\in\mathbb{F}_{p}^{\times}}A_{L}\Big(\frac{(\theta(\tau_{-z}\cdot a),...,\theta(\tau_{z}\cdot a))}{\pi} \Big)-B_{L}\Big(\frac{(\theta(\tau_{-z}\cdot a),...,\theta(\tau_{z}\cdot a))}{\pi} \Big)\\&=p\Delta+O(zC^{z}c_{\mathfrak{F}}^{4}L^{2n+2}\sqrt{p})\\&\geq \frac{1}{2}\Big(\frac{1}{4}\Big)^{z+1}p+O(zC^{z}c_{\mathfrak{F}}^{4}L^{2z+2}\sqrt{p}),
\end{split}
\end{equation}
where in the second step we are using Lemma $\ref{lem : prod2}$, notice that
\begin{itemize}
\item[$i)$] The condition $\tau_{i}\neq\tau_{j}$ if $i\neq j$ is satisfied by definition of acceptable family.
\item[$ii)$] thanks to part $(ii)$ of Lemma $\ref{lem : tec}$ we have that $y$ in Lemma $\ref{lem : prod2}$ is equal to $1$.
\end{itemize}
Let us denote $\delta=1-\varepsilon$ and consider $z=[(\log p)^{\delta}]$. By part $(i)$ of Lemma $\ref{lem : tec}$ we know that $\max (cz,L_{0})\leq L\leq \max (2\gamma cz,L_{0})$, moreover we may assume $ cz\leq L\leq 2\gamma cz$ because $L_{0}$ is an absolute constant (it depends only on $\gamma=\frac{1}{4}$). Then
\begin{equation}
\begin{split}
zC^{z}c_{\mathfrak{F}}^{4}L^{2n+2}\sqrt{p}&\leq (\log p)^{\delta}C^{(\log p)^{\delta}}c_{\mathfrak{F}}^{4}(2\gamma c(\log p)^{\delta})^{(\log p)^{\delta}+2}\sqrt{p}\\&=o((\log p)^{4\delta(\log p)^{\delta}}\sqrt{p})\\&=o(p^{\frac{1}{2}+\eta}),
\end{split}
\end{equation}
for any $\eta>0$. On the other hands,  we have
\[
\Big(\frac{1}{4}\Big)^{z+1}\gg\Big(\frac{1}{4}\Big)^{(\log p)^{\delta}} =e^{-\log(4)(\log p)^{\delta}}=e^{-\log(4)\frac{\log p}{(\log p)^{\varepsilon}}}=p^{-\frac{\log(4)}{(\log p)^{\varepsilon}}}.
\]
Thus we obtain
\[
|S_{p}|\gg_{\epsilon} p^{1-\frac{\log(4)}{(\log p)^{\varepsilon}}},
\]
as we wanted.
\begin{proof}[Proof of Lemma $\ref{lem : tec}$]
The main references for this proof are \cite[Lemma $3.2$]{KLSW} and \cite{BMV}. We define
\[
A_{L}(\mathbf{x}):=\prod_{\substack{i=1\\ i\equiv 1 (2)}}^{z}\alpha_{L,+}(x_{i})\prod_{\substack{i=1\\ i\equiv 1 (2)}}^{1}\alpha_{L,-}(x_{-i}),
\]
where:
\begin{itemize}
\item[$i)$] $\alpha_{L,+}$ is a trigonometric polynomial in one variable of the form
\[
\sum_{|l|\leq L}\alpha_{L,+}(l)e(lx)
\]
defined as in \cite[$(2.2)$ Lemma $5$, $(2.17)$]{BMV} with $u=0$, $v=\frac{1}{2}-\frac{1}{\gamma}$. 
\item[$ii)$] $\alpha_{L,-}$ is a trigonometric polynomial in one variable of the form
\[
\sum_{|l|\leq L}\alpha_{L,-}(l)e(lx)
\]
defined as in \cite[$(2.2)$ Lemma $5$, $(2.17)$]{BMV} with $u=\frac{1}{2}+\frac{1}{\gamma}$, $v=1$.
\end{itemize}
Instead we define
\[
\begin{split}
B_{L}(\mathbf{x})&:=\sum_{\substack{i=1\\i\equiv 1(2)}}^{z}\beta_{L,+}(x_{i})\prod_{\substack{j=1\\j\neq i\\j\equiv 1(2)}}^{z}\alpha_{L,+}(x_{j})\prod_{\substack{j=1\\ j\equiv 1(2)}}^{z}\alpha_{L,-}(x_{-j})\\&+\sum_{\substack{i=1\\ i\equiv 1(2)}}^{z}\beta_{L,-}(x_{-i})\prod_{\substack{j=1\\ j\equiv 1(2)}}^{z}\alpha_{L,+}(x_{j})\prod_{\substack{j=1\\j\neq i\\j\equiv 1(2)}}^{z}\alpha_{L,-}(x_{-j}),
\end{split}
\]
where
\[
\begin{split}
&\beta_{L,+}(x)=\tfrac{1}{2(L+1)}\Big(\sum_{|l|\leq L}\Big(1-\tfrac{|l|}{L+1}\Big)e(lx) + \sum_{|l|\leq L}\Big(1-\tfrac{|l|}{L+1}\Big)e(l(x-\tfrac{1}{2}+\tfrac{1}{\gamma})\Big)\\&
\beta_{L,-}(x)=\tfrac{1}{2(L+1)}\Big(\sum_{|l|\leq L}\Big(1-\tfrac{|l|}{L+1}\Big)e(l(x-\tfrac{1}{2}-\tfrac{1}{\gamma})) + \sum_{|l|\leq L}\Big(1-\tfrac{|l|}{L+1}\Big)e(l(x-1)\Big).
\end{split}
\]
We can rewrite the above trigonometric polynomials as
\[
\begin{split}
&\beta_{L,+}(x)=\tfrac{1}{2(L+1)}\Big(2+\sum_{1\leq l\leq L}\Big(1-\tfrac{l}{L+1}\Big)(\cos(\pi l-\tfrac{ 2\pi l}{\gamma})+\sin(\pi l+\tfrac{ 2\pi l}{\gamma})+1)\cos (2\pi lx)\Big)\\&
\beta_{L,-}(x)=\tfrac{1}{2(L+1)}\Big(2+\sum_{1\leq l\leq L}\Big(1-\tfrac{l}{L+1}\Big)(\cos(-\pi l-\tfrac{2\pi l}{\gamma})+\sin(-\pi l-\tfrac{2\pi l}{\gamma})+1)\cos (2\pi lx)\Big).
\end{split}
\]
Remember that the $n$-th coefficient in the Chebychev expansion of $\alpha_{L,\pm}$ and $\beta_{L,\pm}$ are given by
\[
\int_{0}^{\pi}\alpha_{L,\pm}(\tfrac{\theta}{\pi})U_{n}(\theta)d\mu_{st},\qquad \int_{0}^{\pi}\beta_{L,\pm}(\tfrac{\theta}{\pi})U_{n}(\theta)d\mu_{st},
\]
then part $(iii)$ immediately follows because the above integrals vanishes if $n>2L$. Moreover in \cite[Lemma $5$]{BMV} it is shown that $0\leq\alpha_{L,\pm}(x)\leq 1$ for $x\in [0,1]$ and the same holds for the $|\beta_{L,\pm}|$s by definition. Using Cauchy-Schwarz inequality we get
\[
\Big|\int_{0}^{\pi}\alpha_{L,\pm}U_{n}(\theta)d\mu_{st}\Big|^{2}\leq\int_{0}^{\pi}|\alpha_{L,\pm}(\tfrac{\theta}{\pi})|^{2}d\mu_{st}\cdot\int_{0}^{\pi}|U_{n}(\theta)|^{2}d\mu_{st}\leq 1,
\]
the same argument can be used for $\beta_{L,\pm}$ and this proof part $(ii)$.  It remains to prove only part $(i)$, as we have just observed for any trigonometric polynomial $Y$  the constant term of its  Chebyshev expansion is given by
\[
\int_{0}^{\pi}Y(\theta )d\mu_{st},
\]
so we have that $\Delta$ in part $(i)$ is given by
\[
\begin{split}
\Delta &=\Big(\int_{0}^{\pi}\alpha_{L,+}(\tfrac{\theta}{\pi})d\mu_{st}\Big)^{\frac{z+1}{2}}\Big(\int_{0}^{\pi}\alpha_{L,-}(\tfrac{\theta}{\pi})d\mu_{st}\Big)^{\frac{z+1}{2}}\\&-\frac{z+1}{2}\int_{0}^{\pi}\beta_{L,+}(\tfrac{\theta}{\pi})d\mu_{st}\Big(\int_{0}^{\pi}\alpha_{L,+}(\tfrac{\theta}{\pi})d\mu_{st}\Big)^{\frac{z-1}{2}}\Big(\int_{0}^{\pi}\alpha_{L,-}(\tfrac{\theta}{\pi})d\mu_{st}\Big)^{\frac{z+1}{2}}\\&-\frac{z+1}{2}\int_{0}^{\pi}\beta_{L,-}(\tfrac{\theta}{\pi})d\mu_{st}\Big(\int_{0}^{\pi}\alpha_{L,+}(\tfrac{\theta}{\pi})d\mu_{st}\Big)^{\frac{z+1}{2}}\Big(\int_{0}^{\pi}\alpha_{L,-}(\tfrac{\theta}{\pi})d\mu_{st}\Big)^{\frac{z-1}{2}}.
\end{split}
\]
Using the definition of $\beta_{L,\pm}$ we get
\[
\int_{0}^{\pi}\beta_{L,\pm}(\tfrac{\theta}{\pi})d\mu_{st}=\frac{1}{L+1},
\]
so we can write $\Delta$ as
\[
\begin{split}
\Delta&=\Big(\int_{0}^{\pi}\alpha_{L,+}(\tfrac{\theta}{\pi})d\mu_{st}\Big)^{\frac{z+1}{2}}\Big(\int_{0}^{\pi}\alpha_{L,-}(\tfrac{\theta}{\pi})d\mu_{st}\Big)^{\frac{z+1}{2}}\\&-\frac{z}{2L+2}\Big(\int_{0}^{\pi}\alpha_{L,+}(\tfrac{\theta}{\pi})d\mu_{st}\Big)^{\frac{z-1}{2}}\Big(\int_{0}^{\pi}\alpha_{L,-}(\tfrac{\theta}{\pi})d\mu_{st}\Big)^{\frac{z+1}{2}}\\&-\frac{z}{2L+2}\Big(\int_{0}^{\pi}\alpha_{L,+}(\tfrac{\theta}{\pi})d\mu_{st}\Big)^{\frac{z+1}{2}}\Big(\int_{0}^{\pi}\alpha_{L,-}(\tfrac{\theta}{\pi})d\mu_{st}\Big)^{\frac{z-1}{2}}.
\end{split}
\]
When $L\rightarrow\infty$, $\alpha_{L,\pm}\rightarrow\chi_{\pm\frac{1}{\gamma}}$  in $L^{2}([0,1])$, moreover from \cite[$(2.6)$]{BMV} one has
\[
|\chi_{\pm\frac{1}{\gamma}}-\alpha_{L,\pm}|\leq|\beta_{L,\pm}|\qquad 0\leq x\leq 1
\]
and from the Fourier expansion of $\beta_{L,\pm}(x)$ we have
\[
||\beta_{L,\pm}||^{2}_{L^{2}}\leq\frac{8+3L}{(2L+2)^{2}}\longrightarrow 0,
\] 
thus we have
\[
\int_{0}^{\pi}\alpha_{L,\pm}(\tfrac{\theta}{\pi})d\mu_{st}\longrightarrow\int_{0}^{\pi}\chi_{\pm\frac{1}{\gamma}}(\tfrac{\theta}{\pi})d\mu_{st}=\frac{1}{2}-\frac{1}{\gamma}+\frac{\sin(\tfrac{\pi}{\gamma})\cos (\tfrac{\pi}{\gamma})}{\pi}.
\]
This implies that there exist $L_{0}$, such that the integral in the left hand side of the equation above is $\geq \frac{1}{2}-\frac{1}{\gamma}$ so we get:
\begin{equation}
\Delta\geq\Big(\frac{1}{2}-\frac{1}{\gamma}\Big)^{z-1}\Big(\Big(\frac{1}{2}-\frac{1}{\gamma}\Big)^{2}-\frac{3z}{2L+2}\Big).
\end{equation}
If we assume $2L+2\geq 6z\Big(\frac{1}{2}-\frac{1}{\gamma}\Big)^{-2}$ we get:
\begin{equation}
\Delta\geq\frac{1}{2}\Big(\frac{1}{2}-\frac{1}{\gamma}\Big)^{z+1},
\end{equation}
as we wanted.

\end{proof}

\section{Moments}
\subsubsection*{The Auxiliary Lemma}
Let us start with the following Lemma
\begin{lem}
With the same notation as in Theorem $\ref{thm : tracemom}$, let $0\leq\alpha <\beta\leq 1$, then for any $k\geq 2$ there exist two constant $\gamma,\delta\geq 1$ depending only on $c_{\mathfrak{F}}$, such that
\[
\frac{1}{p-1}\sum_{a\in\mathbb{F}_{p}^{\times}}\Big|\frac{1}{\sqrt{p}}\sum_{\alpha p<x\leq \beta p}t_{a,p}(x)\Big|^{2k}\leq\gamma^{2k}(\log k)^{2k}\Big(\frac{\pi}{\beta -\alpha}\Big)^{-\frac{2k}{\log k}}+\delta^{2k}p^{-\frac{1}{2}}(\log p)^{2k}.
\]
\label{lem : tec1}
\end{lem}
\begin{proof}
Let's start with the quantitative form of the Fourier expansion:
\[
\begin{split}
\frac{1}{\sqrt{p}}\sum_{\alpha p<x\leq \beta p}t_{a,p}&=\frac{1}{2\pi i}\sum_{1\leq |n|\leq p/2}\frac{K_{a,p}(n)}{n}(1-e((\beta-\alpha)n))e(\alpha n)\\&+(\beta-\alpha)K_{a,p}(0)+O(1)
\end{split}
\]
To simplify the notation, for any $-p/2 \leq n\leq p/2$ we define
\[
c_{n}:=\frac{(1-e((\beta-\alpha)n))e(\alpha n)}{n},
\]
so we can write the equation above as
\[
\frac{1}{\sqrt{p}}\sum_{\alpha p<x\leq \beta p}t_{a,p}(x)=\frac{1}{2\pi i}\sum_{1<|n|<p/2}K_{a,p}(n)c_{n}+(\beta-\alpha)K_{a,p}(0)+O(1).
\]
By the triangular inequality one gets
\[
\begin{split}
\frac{1}{p-1}\sum_{a\in\mathbb{F}_{a}^{\times}}\Big|\frac{1}{\sqrt{p}}\sum_{\alpha p<x\leq \beta p}t_{a,p}(x)\Big|^{2k}&\leq\frac{1}{(p-1)\pi^{2k}}\sum_{a\in\mathbb{F}_{p}^{\times}}\Big|\sum_{1< |n|< p/2}K_{a,p}(n)c_{n}\Big|^{2k}\\&+O(2^{4k}+2^{4k}(\beta-\alpha)^{2k}c_{\mathfrak{F}}^{2k})\\&=\frac{1}{(p-1)\pi^{2k}}\sum_{a\in\mathbb{F}_{p}^{\times}}\Big|\sum_{1< |n|< p/2}K_{1,p}(\tau_{n}\cdot a)c_{n}\Big|^{2k}\\&+O(2^{4k}+2^{4k}(\beta-\alpha)^{2k}c_{\mathfrak{F}}^{2k}),
\end{split}
\]
where in the first inequality we use the fact that $K_{a,p}(0)\leq c (\mathcal{F}_{a,p})\leq c_{\mathfrak{F}}$ by hypothesis.
To conclude the proof of the Lemma it is enough to provide a bound for the first term in the right hand side. Extending the $2k$-power we get
\[
\begin{split}
\sum_{a\in\mathbb{F}_{p}^{\times}}\sum_{n_{1}}\cdots\sum_{n_{k}}\sum_{l_{1}}\cdots\sum_{l_{k}}&K_{1,p}(\tau_{n_{1}}\cdot a)\cdots K_{1,p}(\tau_{n_{k}}\cdot a)\cdot\\&\cdot K_{1,p}(\tau_{l_{1}}\cdot a)\cdot...\cdot K_{1,p}(\tau_{l_{k}}\cdot a)c_{n_{1}}\cdots c_{n_{k}}\overline{c_{l_{1}}}\cdots\cdot\overline{c_{l_{k}}}.
\end{split}
\]
An application of \cite[Corollary $1.7$]{FKM3} implies 
\[
\Big|\sum_{a\in\mathbb{F}_{p}^{\times}}K_{1,p}(\tau_{n_{1}}\cdot a)\cdot...\cdot K_{1,p}(\tau_{n_{k}}\cdot a)K_{1,p}(\tau_{l_{1}}\cdot a)\cdot...\cdot K_{1,p}(\tau_{l_{k}}\cdot a)-m(\textbf{n},\textbf{l})p\Big|\leq\delta_{1}^{2k}\sqrt{p},
\]
where the constant $\delta_{1}$ depends only on $c_{\mathfrak{F}}$ and moreover $m(\textbf{n},\textbf{l})\neq 0$ if and only if every entries of the array $(\textbf{n},\textbf{l})$ have even multiplicity (see \cite[Corollary $1.7$]{FKM3} for a precise definition of $m(\textbf{n},\textbf{l})$). Thanks to this we get
\[
\sum_{a\in\mathbb{F}_{p}^{\times}}\Big|\sum_{1<|n|<p/2}K_{a,p}(n)c_{n}\Big|^{2k}\leq A+B,
\]
where
\[
A:=p\Big|\sum_{\textbf{n},\textbf{l}}c_{n_{1}}\cdot...\cdot c_{n_{k}}\overline{c_{l_{1}}}\cdot...\cdot\overline{c_{l_{k}}}m(\textbf{n},\textbf{l})\Big|,
\]
and
\[
B:=\Big|\sqrt{p}\delta_{1}^{2k}\sum_{\textbf{n},\textbf{l}}c_{n_{1}}\cdot...\cdot c_{n_{k}}\overline{c_{l_{1}}}\cdot...\cdot\overline{c_{l_{k}}}\Big|.
\]
Let us bound $B$ first.
\[
B\leq \sqrt{p}\delta_{1}^{2k}\sum_{\textbf{n},\textbf{l}}|c_{n_{1}}|\cdot...\cdot |c_{n_{k}}||c_{l_{1}}|\cdot...\cdot |c_{l_{k}}|=\sqrt{p}\delta_{1}^{2k}\Big(\sum_{n}|c_{n}|\Big)^{2k}.
\]
On the other hand
\[
|c_{n}|\leq2\min \Big(\frac{1}{n},\frac{\beta-\alpha}{\pi}\Big)\leq\frac{2}{n},
\]
hence we get $B\leq\sqrt{p}\delta_{2}^{2k}(\log p)^{2k}$ for some $\delta_{2}>0$ depending only on $c_{\mathfrak{F}}$.\newline
To bound $A$ we can proceed as follows: first observe that if $m(\textbf{n},\textbf{l})\neq 0$ then there exists a constant dependent only on $c_{\mathfrak{F}}$, let's say $\gamma_{1}$, such that $m(\textbf{n},\textbf{l})\leq\gamma_{1}^{2k}$ (again the reference here is \cite[Corollary $1.7$]{FKM3}). Thus
\begin{equation}
A\leq \gamma_{1}^{2k}p\sum_{n}\sum_{(n_{1},...,n_{2k})\in m(n)}c_{n_{1}}\cdots c_{n_{2k}},
\label{eq : boundA}
\end{equation}
where $m(n):=\{(n_{1},...,n_{2k}):n_{1}\cdots n_{2k}=n\quad\text{any $n_{i}$ appears an even number of times}\}$. On the other hand we have that
\[
c_{n_{1}}\cdots c_{n_{2k}}\leq 2^{2k}\min \Big(\frac{1}{n},\Big(\frac{\beta-\alpha}{\pi}\Big)^{2k}\Big)=:b(n).
\]
Let us focus our attention on the size of $|m(n)|$. First observe that by definition, $|m(n)|=0$ is $n$ is not a square. Moreover for any $(n_{1},...,n_{2k})\in m(n^{2})$ we can find two set $S_{1},S_{2}\subset\{1,...,2k\}$ such that
\[
|S_{1}|=|S_{2}|=k\qquad S_{1}\cap S_{2}=\emptyset\qquad n=\prod_{i\in S_{1}}n_{i}=\prod_{i\in S_{2}}n_{i},
\] 
thus
\[
|m(n^{2})|\leq\binom{2k}{k}d_{k}(n)^{2}.
\]
Inserting this in equation $(\ref{eq : boundA})$ we get
\[
\begin{split}
A&\leq\gamma_{1}^{2k}\binom{2k}{k}p\sum_{n}d_{k}(n)^{2}b(n^{2})\\&\leq \gamma_{1}^{2k}\binom{2k}{k}p\sum_{n\leq p}d_{k}(n)^{2}b(n^{2})+O_{k,\varepsilon}(p^{\varepsilon}).
\end{split}
\]
On the other hand is shown in \cite[Lemma $4.1$]{BG} that
\[
\sum_{n\leq p}d_{k}(n)^{2}b(n^{2})\leq 2^{k}(\log k)^{2k}\Big(\frac{\pi}{\beta -\alpha}\Big)^{-\frac{2k}{\log k}}
\]
and this conclude the proof.
\end{proof}

\subsubsection*{Proof of Theorem \ref{thm : tracemom}}
We are finally ready to prove our result on moments
\begin{proof}[Proof of Theorem \ref{thm : tracemom}] Let us start with the lower bound. By Lemma $\ref{lem : ker}$ we have that
\[
M(t_{a,p})\geq \frac{1}{4\pi}\Big|\sum_{1\leq n\leq k}\frac{K_{1,p}(\tau_{n}\cdot a)}{n}(1-e(\alpha n))\Big|+O_{c(\mathfrak{F})}(1)
\]
for any $p$ large enough and any $a\in\mathbb{F}_{p}^{\times}$. On the other hand $\mathcal{K}_{1,p}$ is a bountiful sheaf, so the sheaves $\mathcal{K}_{1,p},\tau_{-1}^{*}\mathcal{K}_{1,p},...,\tau_{-k}^{*},\mathcal{K}_{1,p},\tau_{k}^{*}\mathcal{K}_{1,p}$ satisfy the Goursat-Kolchin-Ribet criterion and so
\[
a\mapsto (K_{1,p}(\tau_{1}\cdot a),K_{1,p}(\tau_{-1}\cdot a),...,K_{1,p}(\tau_{-k}\cdot a),K_{1,p}(\tau_{k}\cdot a))
\]
become equidistributed in $(\prod_{i=1}^{2k}G^{\geom},\mu_{G^{\geom}}^{\otimes 2k})$ when $p\rightarrow\infty$. Now if we define
\[
\begin{split}
S_{k,p}:=\{a\in\mathbb{F}_{p}^{\times}:& K_{1,p}(\tau_{i}\cdot a)>\sqrt{2}\text{ }\forall 0<i<k, i\equiv 1(2),\\&  K_{1,p}(\tau_{-i}\cdot a)<-\sqrt{2}\text{ }\forall 0<i<k, i\equiv 1(2))\},
\end{split}
\]
we have
\[
M(t_{a,p})\geq\Big(\frac{1}{\sqrt{2}\pi}+o(1)\Big)\log k
\]
for any $a\in S_{k,p}$. Hence
\[
\frac{1}{p-1}\sum_{a\in\mathbb{F}_{p}^{\times}}M(t_{a,p})^{2k}\geq \frac{1}{p-1}\sum_{a\in S_{k,p}}M(t_{a,p})\geq \Big(\frac{c}{\sqrt{2}\pi}+o(1)\Big)^{2k}(\log k)^{2k},
\]
where $c=|\{a\in\mathbb{F}_{p}^{\times}: K_{1,p}( a)>\sqrt{2}\}|$ which depends only on $c_{\mathfrak{F}}$. Let us prove now the upper bound.
For any $a\in\mathbb{F}_{q}^{\times}$ let $N_{a,p}$ be the smallest integer such that
\[
M(t_{a,p})=\Big|\frac{1}{\sqrt{p}}\sum_{x\leq N_{a,p}}t_{a,p}(x)\Big|.
\]
At this point we would like to apply Lemma $\ref{lem : tec1}$ but the issue here is that the $N_{a,p}$s could be very different each other. To turn around to this problem, following the same strategy as in \cite{MV2} and \cite{BG}, we will use the Rademacher-Menchov trick: first of all write $\frac{N_{a,p}}{p}$ in base two
\[
\frac{N_{a,p}}{p}=\sum_{j=1}^{\infty}a_{j}2^{-j}\qquad a_{j}\in\{0,1\},
\]
and let $N_{a,p}(L_{p})/p$ be the truncation of this series to the factor of power $L_{p}$. Then we have
\[
M(t_{a,p})\leq\Big|\frac{1}{\sqrt{p}}\sum_{x\leq N_{a,p}(L_{p})}t_{a,p}(x)\Big|+\Big|\frac{1}{\sqrt{p}}\sum_{N_{a,p}(L_{p})<x\leq N_{a,p}}t_{a,p}(x)\Big|,
\]
notice that the length of the second summation is $\leq\frac{p}{2^{L_{p}}}$, let's denote $E(a,p,L_{p})$ this sum. An application of the H\"older inequality leads to
\[
\begin{split}
M(t_{a,p})^{2k}\\&\leq 2^{2k}\Big(\sum_{l\leq L_{p}}\frac{1}{l^{\frac{2k\alpha}{2k-1}}}\Big)^{2k-1}\Big(\sum_{l\leq L_{p}}l^{2k\alpha}\Big|\frac{1}{\sqrt{p}}\sum_{N_{a,p}(l)<x\leq N_{a,p}(l+1)}t_{a,p}(x)\Big|^{2k}\Big)\\&+2^{2k}E(a,p,L_{p})^{2k}.
\end{split}
\]
at this point we observe at first that $N_{a,p}(l+1)\leq N_{a,p}(l)+p2^{-(l+1)}$ , moreover for the value of $N_{a,p}(l)$ one has $2^{l-1}$ possibility so
\[
\begin{split}
M(t_{a,p})^{2k}&\leq 2^{2k}\Big(\sum_{l\leq L_{p}}\frac{1}{l^{\frac{2k\alpha}{2k-1}}}\Big)^{2k-1}\Big(\sum_{l\leq L_{p}}l^{2k\alpha}\sum_{0\leq m\leq 2^{l}-1}\Big|\frac{1}{\sqrt{p}}\sum_{p\frac{m}{2^{l}}<x\leq p(\frac{m}{2^{l}}+2^{-(l+1)})}t_{a,p}(x)\Big|^{2k}\Big)\\&+2^{2k}E(a,p,L_{p})^{2k}.
\end{split}
\]
We can now apply the Lemma $\ref{lem : tec1}$ and choose $\alpha=3/2$ getting
\[
\begin{split}
\frac{1}{p-1}\sum_{a\in\mathbb{F}_{p}^{\times}}M(t_{a,p})^{2k}&\leq  2^{2k}\Big(\sum_{l\leq L_{p}}\frac{1}{l^{\frac{3k}{2k-1}}}\Big)^{2k-1}\Big(\sum_{l\leq L_{p}}l^{3k}2^{l}(\gamma^{2k}(\log k)^{2k}2^{-\frac{kl}{\log k}}+\\&+\delta^{2k}p^{-\frac{1}{2}}(\log p)^{2k})\Big) +\frac{2^{2k}}{p-1}\sum_{a\in\mathbb{F}_{p}^{\times}}E(a,p,L_{p})^{2k}.
\end{split}
\]
Let $L_{p}:=\log_{2}\Big(\frac{p^{\frac{1}{2}}}{(\log p)^{8k}}\Big)$, then
\[
\begin{split}
(2\delta)^{2k}\Big(\sum_{l\leq L_{p}}\frac{1}{l^{\frac{3k}{2k-1}}}\Big)^{2k-1}\sum_{l\leq L_{p}}l^{3k}2^{l}p^{-\frac{1}{2}}(\log p)^{2k}&\ll_{k}(\log p)^{8k}2^{L}p^{-\frac{1}{2}}\\&\ll_{k}1,
\end{split}
\]
where in the first step we are using that
\[
\sum_{l}\frac{1}{l^{\frac{3k}{2k-1}}}\ll 1.
\]
Moreover using the inequality
\[
\sum_{l\leq L_{p}}l^{3k}2^{l}2^{-\frac{kl}{\log k}}\leq\text{exp}(3k\log\log k+O(k))
\]
proven in \cite[Theorem $1.1$]{BG} we get
\[
(2\gamma)^{2k}(\log k)^{2k}\Big(\sum_{l\leq L_{p}}\frac{1}{l^{\frac{2k}{2k-1}}}\Big)^{2k-1}\sum_{l\leq L_{p}}l^{2k}2^{l}2^{-\frac{kl}{\log k}}\ll_{k}1.
\]
On the other hand we have already pointed out that the length of $E(a,p,L_{p})$ is at most $\frac{p}{2^{L_{p}}}=p^{\frac{1}{2}}(\log p)^{8k}$, so thanks to \cite[Theorem $1.1$]{FKMRRS} we have
\[
|E(a,p,L_{p})|\leq 16\log (4e^{8}(\log p)^{k})
\]
then
\[
\frac{1}{p-1}\sum_{a\in\mathbb{F}_{p}^{\times}}E(a,p,L_{p})^{2k}\leq (Ck)^{k}(\log\log p)^{2k}.
\]
Now let us assume that
\begin{equation}
\Big|\sum_{N\leq x\leq N+H }t_{a,p}(x)\Big|\ll_{c_{\mathfrak{F}}} H^{1-\varepsilon}
\label{eq : cond2}
\end{equation}
holds uniformly for any $1<N<p$, $p^{1/2-\varepsilon/2}<H<p^{1/2+\varepsilon/2}$ and $a\in\mathbb{F}_{p}^{\times}$. Starting again from
\[
\begin{split}
\frac{1}{p-1}\sum_{a\in\mathbb{F}_{p}^{\times}}M(t_{a,p})^{2k}\leq & 2^{2k}\Big(\sum_{l\leq L_{p}}\frac{1}{l^{\frac{2\alpha k}{2k-1}}}\Big)^{2k-1}\Big(\sum_{l\leq L_{p}}l^{2k\alpha}2^{l}(\gamma^{2k}(\log k)^{2k}2^{-\frac{kl}{\log k}}+\\&+\delta^{2k}p^{-\frac{1}{2}}(\log p)^{2k})\Big) +\frac{2^{2k}}{p-1}\sum_{a\in\mathbb{F}_{p}^{\times}}E(a,p,L_{p})^{2k}.
\label{eq : momproof}
\end{split}
\]
We can now choose $L_{p}:=\frac{1-\varepsilon}{2}\log_{2} p$, then $(\ref{eq : cond2})$ implies
\[
E(a,p,L_{p})=\Big|\frac{1}{\sqrt{p}}\sum_{N_{a,p}(L_{p})<x\leq N_{a,p}}t_{a,p}(x)\Big|\ll_{c_{\mathfrak{F}}} p^{-\varepsilon'}.
\]
Then
\[
\frac{1}{p-1}\sum_{a\in\mathbb{F}_{p}^{\times}}E(a,p,L_{p})\ll_{c_{\mathfrak{F}}} p^{-\varepsilon'},
\]
for some $\varepsilon'>0$. Moreover observe that
\[
\begin{split}
p^{-\frac{1}{2}}(\delta\log p)^{2k}\Big(\sum_{l\leq L_{p}}l^{2k\alpha}2^{l}\Big)&\leq p^{-\frac{1}{2}}(\delta\log p)^{2k}\Big(\sum_{l\leq L_{p}} L_{p}^{4k}2^{l}\Big)\\&\ll_{k,c_{\mathfrak{F}}}p^{-\frac{1}{2}}(\log p)^{2k}2^{L_{p}}\\&\ll p^{-\varepsilon/2}(\log p)^{2k}.
\end{split}
\]
So we get
\[
\begin{split}
\frac{1}{p-1}\sum_{a\in\mathbb{F}_{p}^{\times}}M(t_{a,p})^{2k}\leq & (2\gamma\log k)^{2k}\Big(\sum_{l\leq L_{p}}\frac{1}{l^{\frac{2\alpha k}{2k-1}}}\Big)^{2k-1}\Big(\sum_{l\leq L_{p}}l^{2k\alpha}2^{l}2^{-\frac{kl}{\log k}}\Big)\\& +O_{k,c_{\mathfrak{F}}}(p^{-\varepsilon''}),
\end{split}
\]
$\varepsilon''>0$. On the other hand we have (\cite[Theorem $1.1$]{BG})
\[
\sum_{l\leq L_{p}}\frac{1}{l^{\frac{2\alpha k}{2k-1}}}\leq (\alpha -1)^{1-2k},\qquad \sum_{l\leq L_{p}}l^{2k\alpha}2^{l}2^{-\frac{kl}{\log k}}\leq\exp (2k\alpha\log\log k +O(k)),
\]
so choosing $\alpha=1+1/\log\log k$ we get the result.
\end{proof}
we conclude with
\begin{proof}[Proof of Corollary $\ref{cor : quan}$.] 
For $(i)$ observe that from the proof of lower bond of Theorem $\ref{thm : tracemom}$ it follows that any element, $a$, in the set
\[
\begin{split}
S_{h,p}:=\{a\in\mathbb{F}_{p}^{\times}:& K_{1,p}(\tau_{i}\cdot a)>\sqrt{2}\text{ }\forall 0<i<h, i\equiv 1(2),\\&  K_{1,p}(\tau_{-i}\cdot a)<-\sqrt{2}\text{ }\forall 0<i<h, i\equiv 1(2)\}
\end{split}
\]
is such that $M(t_{a,p})>\text{const}\cdot\log h$. Moreover we have that $|S_{h,p}|>c^{2h}$ for some constant $0<c<1$ depending on $c_{\mathfrak{F}}$. Choosing $h=\exp((\text{const})^{-1}\cdot A)$ we get 
\[
D_{\mathfrak{F}}(A)\geq |S_{\exp((\text{const})^{-1}\cdot A)}|.
\]
To conclude, the proof of $(ii)$ is exactly the same as in \cite[Theorem $1.3$]{BG}.
\end{proof}

\bibliographystyle{alpha}

\bibliography{bibtesi.bib}

\end{document}